%% file: FN.tex
\documentclass[graybox]{svmult}

\usepackage{mathptmx}       
\usepackage{helvet}         
\usepackage{courier}        
\usepackage{type1cm}        
\usepackage{makeidx}         
\usepackage{graphicx}        
\usepackage{multicol}        
\usepackage[bottom]{footmisc}

\makeindex             

\begin{document}

\title*{On stationary solutions of KdV and mKdV equations}
\author{A.V.~Faminskii and A.A.~Nikolaev}
\institute{A.V.~Faminskii \at Peoples' Friendship University of Russia, Miklukho-Maklai str. 6, Moscow, 117198, Russia, \email{afaminskii@sci.pfu.edu.ru}
\and A.A.~Nikolaev \at Peoples' Friendship University of Russia, Miklukho-Maklai str. 6, Moscow, 117198, Russia, \email{anikolaev.rudn@yandex.ru}}
%
%
\maketitle

\abstract*{Stationary solutions on a bounded interval for an initial-boundary value problem to Korteweg--de~Vries and modified Korteweg--de~Vries equation (for the last one both in focusing and defocusing cases) are constructed. The method of the study is based on the theory of conservative systems with one degree of freedom. The obtained solutions turn out to be periodic. Exact relations between the length of the interval and coefficients of the equations which are necessary and sufficient for existence of nontrivial solutions are established.}

\abstract{Stationary solutions on a bounded interval for an initial-boundary value problem to Korteweg--de~Vries and modified Korteweg--de~Vries equation (for the last one both in focusing and defocusing cases) are constructed. The method of the study is based on the theory of conservative systems with one degree of freedom. The obtained solutions turn out to be periodic. Exact relations between the length of the interval and coefficients of the equations which are necessary and sufficient for existence of nontrivial solutions are established.}

\section*{}
Both Korteweg--de~Vries equation (KdV)
$$
u_t+au_x+u_{xxx}+uu_x=0
$$
and modified Korteweg--de~Vries equation (mKdV)
$$
u_t+au_x+u_{xxx}\pm u^2u_x=0
$$
(the sign "+" stands for the focusing case and the sign "-" -- for the defocusing one) describe propagation of long nonlinear waves in dispersive media. We assume $a$ to be an arbitrary real constant. If these equations are considered on a bounded interval $(0,L)$, then for well-posedness of an inital-boundary value problem besides an initial profile one must set certain boundary conditions, for example, 
$$
u\big|_{x=0} = u\big|_{x=L} =u_x\bigl|_{x=L}=0
$$
(see \cite{Kh, F01, BSZ, F07} and others).

It follows from the results of \cite{FL} that such a problem for KdV equation possesses certain internal dissipation:
under some relations between $a$ and $L$ and sufficiently small initial data solutions decay at large time. Similar properties hold for mKdV equation. In order to answer the question if the smallness is essential one has to construct non-decaying solutions. The simplest case of such solutions are stationary solutions: $u=u(x)$. In this situation the considered equations are reduced to the following ordinary differential equations:
\begin{equation}\label{1}
u''' + au' + uu' =0,
\end{equation}
\begin{equation}\label{2}
u''' + au' + u^2u' =0,
\end{equation}
\begin{equation}\label{3}
u''' + au' - u^2u' =0,
\end{equation}
and the boundary conditions --- to the following ones:
\begin{equation}\label{4}
u(0)=u(L)=u'(L)=0.
\end{equation}

The goal of the present paper is to investigate existence of nontrivial solutions to these problems under different relations between $a$ and $L$. The method of the study is based on the qualitative theory of conservative systems with one degree of freedom (see, for example, \cite{A}).

The first example of such a solution by this method for equation (\ref{1}) was constructed in the case $a=0$, $L=2$ in \cite{GS}. In recent paper \cite{DNa} also for equation (\ref{1}) such solutions were constructed for $a=1$ and $L\in (0,2\pi)$ and exact formulas via elliptic Jacobi functions were obtained. In the present paper these special functions are not used.

\begin{lemma}\label{L1}
If $u\in C^3[0,L]$ is a solution to any problems (\ref{1}), (\ref{4}), or (\ref{2}), (\ref{4}), or (\ref{3}), (\ref{4}), then it is infinitely smooth and periodic with period $L$.
\end{lemma}

\begin{proof}
Integrating each of the equations (\ref{1})--(\ref{3}) we obtain that the function $u$ satisfies an equation
\begin{equation}\label{5}
u'' + F'(u) =0, \quad F(0)=0, \quad F\in C^\infty.
\end{equation}
Following \cite{A} introduce a "full energy" $E(x) \equiv \frac12 \bigl(u'(x)\bigr)^2 +F\bigl(u(x)\bigr)$. Then (\ref{5}) yields that $E'(x)\equiv 0$, that is $E(x)\equiv \mathrm{const}$. By virtue of (\ref{4}) $E(L)=0$, therefore $E(0)=0$ and so $u'(0)=0$. The end of the proof is obvious.
\qed
\end{proof}

Further let a fundamental period for a nontrivial periodic function denotes a minimal possible positive value of a period.  

By a symbol $u_{a,T}$ denote a nontrivial solution to any of considered problems with the fundamental period $T$.  

\begin{theorem}\label{T1}
If $aL^2 \ne 4\pi^2$ then there exists a unique solution $u_{a,L}$ to problem (\ref{1}), (\ref{4}). If $aL^2 = 4\pi^2$ such a solution does not exist.
\end{theorem}

\begin{theorem}\label{T2}
If $aL^2 < 4\pi^2$ then there exists a unique up to the sign solution $u_{a,L}$ to problem (\ref{2}), (\ref{4}). 
If $aL^2 \geq 4\pi^2$ such solutions do not exist.
\end{theorem}

\begin{theorem}\label{T3}
If $aL^2 > 4\pi^2$ then there exists a unique up to the sign solution $u_{a,L}$ to problem (\ref{3}), (\ref{4}). 
If $aL^2 \leq 4\pi^2$ such solutions do not exist.
\end{theorem}

\begin{remark}\label{R1}
If $aL^2 \ne 4\pi n^2$ for certain natural $n\geq 2$ then obviously  the function $u(x) \equiv n^2 u_{a/n^2,L}(nx)$ is a solution to problem (\ref{1}), (\ref{4}) with the fundamental period $T=L/n$. If $aL^2 < 4\pi n^2$ for certain natural $n$ then the function $u(x) \equiv n u_{a/n^2,L}(nx)$ is a solution to problem (\ref{2}), (\ref{4}) with the fundamental period $T=L/n$.
In particular, nontrivial solutions to problems (\ref{1}), (\ref{4}) and (\ref{2}), (\ref{4}) exist for any $a$ and positive $L$.
If $aL^2 \leq 4\pi^2$ then nontrivial solutions to problem (\ref{3}), (\ref{4}) do not exist.
\end{remark}

Further for convenience we pass from the segment $[0,L]$ to the segment $[-1,1]$. For $x\in [-1,1]$ in the case of equation (\ref{1}) make a substitution $y(x) \equiv \frac {L^2}4 u\bigl(\frac L2 (x+1)\bigr)$, while in the case of equations (\ref{2}) and (\ref{3}) --- substitution $y(x) \equiv \frac L2 u\bigl(\frac L2 (x+1)\bigr)$. Then for $b=\frac{L^2}4 a$ these equations transform respectively to the following ones:
\begin{equation}\label{6}
y''' + by' + yy' =0,
\end{equation}
\begin{equation}\label{7}
y''' + by' + y^2y' =0,
\end{equation}
\begin{equation}\label{8}
y''' + by' - y^2y' =0,
\end{equation}
and consider periodic solutions to these equations with the fundamental period $T=2$ such that
\begin{equation}\label{9}
y(-1) = y'(-1) =0.
\end{equation}

We apply the following lemma in the spirit of the qualitative theory of conservative systems with one degree of freedom.

\begin{lemma}\label{L2}
Consider an initial value problem
\begin{equation}\label{10}
y'' + F'(y) =0, \qquad y(-1) = y'(-1) =0,
\end{equation}
where $F\in C^\infty$, $F(0)=0$. Then a nontrivial periodic solution to problem (\ref{10}) with the fundamental period $T=2$ exist if and only if $F'(0) \ne 0$ and there exists $y_0 \ne 0$ such that $F(y_0)=0$, $F'(y_0) \ne 0$, $F(y)<0$ for $y\in (0,y_0)$ if $y_0>0$, $F(y)<0$ for $y\in (y_0,0)$ if $y_0<0$ and
\begin{equation}\label{11}
\int_0^{y_0} \frac{dy}{\sqrt{-2F(y)}} =1 \quad \mathrm{if}\  y_0>0,\qquad
\int_{y_0}^0 \frac{dy}{\sqrt{-2F(y)}} =1 \quad \mathrm{if}\ y_0<0.
\end{equation}
\end{lemma}

\begin{proof}
First of all note that similarly to (\ref{5}) $E(x)\equiv \frac12 \bigl(y'(x)\bigr)^2 +F\bigl(y(x)\bigr) \equiv 0$ if $y(x)$ is a solution to problem (\ref{10}). Due to uniqueness of solutions to the initial value problem the condition $F'(0) \ne 0$ is necessary for existence of nontrivial solutions.

Consider, for example, the case $F'(0)<0$. If the function $F$ is negative $\forall y>0$ then it is easy to see that there is no periodic solution to problem (\ref{10}). Therefore, existence of positive $y_0$ such that $F(y_0)=0$, $F(y)<0$ for $y\in (0,y_0)$ is necessary.

Uniqueness of the solution implies that the function $y(x)$ is even (if exists). Then it is easy to see that it possesses the following properties: $y'(x)>0$ for $x\in (-1,0)$, $y'(x)<0$ for $x\in (0,1)$, $y(0)=y_0$, $y'(0)=0$. Again due to uniqueness $F'(y_0)\ne 0$.

Therefore, for $x\in [0,1]$ the function $y(x)$ satisfies the following conditions:
$$
\frac{dy}{dx} = -\sqrt{-2F(y)},\quad y(0)=y_0, \quad y(1)=0.
$$
Integrating we obtain that $\displaystyle\int_0^{y_0} \frac{dy}{\sqrt{-2F(y)}} =1$.

It is easy to see that under these assumptions the desired solution exist. The case $F'(0)>0$ is considered in a similar way (then $y_0<0$).
\qed
\end{proof}

Now we can prove our theorems.

\begin{proof}[Theorem~\ref{T1}]
Equation (\ref{6}) is equivalent to equation
\begin{equation}\label{12}
y'' + by + \frac 12 y^2 =c
\end{equation}
for certain real constant $c$. Therefore, construction of a solution transforms to search of a constant $c$ such for a function
$$
F(y) \equiv \frac16 y^3 + \frac{b}2 y^2 -cy = \frac 16 y(y^2+3by-6c) \equiv \frac16 y F_0(y)
$$
the hypothesis of Lemma~\ref{L2} is satisfied. Note that $F'(y) = \frac12 y^2 +by-c$. Therefore, the condition $F'(0)\ne 0$ implies that $c\ne 0$.

Real simple nonzero roots of the function $F_0$ exist if and only if $D=9b^2+24c>0$ and then these roots are expressed by formulas $y_0 = \frac12 (-3b +\sqrt{D})$, $y_1 = -\frac12 (3b+\sqrt{D})$. 

It is easy to see that if $c>0$ then for any $b$ the root $y_0>0$, $F(y)<0$ for $y\in (0,y_0)$, $F'(y_0)\ne 0$. If $c\in (-3b^2/8,0)$ then for $b>0$ the root $y_0<0$,  $F(y)<0$ for $y\in (y_0,0)$, $F'(y_0)\ne 0$.

Therefore, we have to find the constant $c$ for which condition (\ref{11}) is satisfied. Note that
$$
-2F(y) = \frac 13 y (y_0-y)(y-y_1).
$$
After the change of variable $y=y_0t$ each of equations (\ref{11}) reduces to an equation
$$
I(b,c) \equiv \sqrt{3} \int_0^1 \frac{dt}{\sqrt{t(1-t)(y_0t-y_1)}} =1.
$$
Since $y_0t-y_1 = \frac12(\sqrt{D}-3b)t+\frac12(\sqrt{D}+3b)$ it is easy to see that for the fixed $b$ the function $I(b,c)$ monotonically decreases. Moreover, $\displaystyle\lim\limits_{c\to+\infty} I(b,c)=0$ and for $b>0$
$$
\lim\limits_{c\to -\frac38 b^2+0} I(b,c) =\sqrt{\frac2b}\int_0^1 \frac{dt}{\sqrt{t}(1-t)}=+\infty,\
\lim\limits_{c\to 0} I(b,c) =\frac1{\sqrt{b}}\int_0^1 \frac{dt}{\sqrt{t(1-t)}}=\frac\pi{\sqrt{b}},
$$
for $b= 0$
$$
\lim\limits_{c\to 0+0} I(b,c) = \lim\limits_{c\to 0+0}\frac1{\sqrt{2c}}\int_0^1 \frac{dt}{\sqrt{t(1-t)(t+1)}} =+\infty,
$$
for $b< 0$
$$
\lim\limits_{c\to 0+0} I(b,c) = \frac1{\sqrt{|b|}}\int_0^1 \frac{dt}{t\sqrt{1-t}} =+\infty.
$$
Therefore, the desired value of $c$ exists and is unique if $b\ne \pi^2$, while for $b=\pi^2$ such a value does not exist.
\qed
\end{proof}

\begin{remark}\label{R2}
The substitution $u(x)=a_0+v(x-x_0)$ under the appropriate choice of the parameters $a_0$ and $x_0$ transforms any periodic solution of equation (\ref{1}) with the period $L$ to solution of an equation $v'''+(a+a_0)v'+vv'=0$ satisfying conditions $v(0)=v'(0)=v(L)=v'(L)=0$. Therefore, any solution of equation (\ref{1}) with the fundamental period $L$ can be expressed in this way by the functions $u_{a+a_0,L}$. Solutions similar to functions $u_{a,L}$ were considered also in \cite{Ne}. In \cite{AnBS} representation of periodic solutions of equation (1) is given via elliptic Jacobi functions. The advantage of our approach is that it can give transparent description of solutions.

Consider, for example, the case $b>0$. Then for $b\in (0,\pi^2)$ the constructed solution of problem (\ref{6}), (\ref{9}) is an even "hill" of the height $y_0 =\frac 12(-3b+\sqrt{9b^2+24c})>0$, while for $b>\pi^2$ --- an even "hole" of the depth $y_0<0$. Note that $I_c(b,c)<0$, $I_b(b,c)<0$. Therefore, the equation $I(b,c)=1$ determines a smooth decreasing function $c(b)$. Since $I(\pi^2,0)=1$ we have that $c(\pi^2)=0$. Return to equation (\ref{1}). Let $a>0$. If
$u_0=\frac12(-3a+\sqrt{9a^2 +384c L^{-2}})$, where $c=c(L^2a/4)$, then for $L<2\pi/\sqrt{a}$ the solution $u_{a,L}$ to problem (\ref{1}), (\ref{4}) is a "hill" of the height $u_0>0$ and for $L>2\pi/\sqrt{a}$ --- a "hole" of the depth $u_0<0$ (the center in both cases is at the point $L/2$). In addition, $u_0\to+\infty$ as $L\to 0$, $u_0\to 0$ as $L\to 2\pi/\sqrt{a}$, $u_0\to 0$ as $L\to +\infty$.
\end{remark}

\begin{proof}[Theorem~\ref{T2}]
Equation (\ref{7}) is equivalent to equation
\begin{equation}\label{13}
y'' + by + \frac 13 y^3 =c
\end{equation}
for certain real constant $c$. Let
$$
F(y) \equiv \frac1{12} y^4 + \frac{b}2 y^2 -cy = \frac 1{12} y(y^3+6by-12c) \equiv \frac1{12} y F_0(y).
$$
Note that the substitution $z(x)\equiv -y(x)$ leads to an equation similar to (\ref{13}), where $c$ is replaced by $(-c)$. Therefore, further it is sufficient to assume that $c>0$ (if $c=0$ then $F'(0)=0$). 

Similarly to the proof of Theorem~\ref{T1} we need to find the roots of the function $F_0$. We apply Cardano formulas. Let
$D=8b^3+36c^2$, 
$$
p= \sqrt[3]{6c +\sqrt{D}},\quad q= \sqrt[3]{6c -\sqrt{D}} \quad \mathrm{if}\ D\geq 0,
$$
$$
p= \sqrt[3]{6c +i\sqrt{|D|}}= \sqrt{2|b|}e^{\frac{i}3\arccos(3c/\sqrt{2|b|^3})},\quad q=\overline{p} \quad 
\mathrm{if}\ D<0.
$$
The the function $F_0$ has a real root $y_0=p+q>0$. Moreover, if $D>0$ there are two complex conjugate roots with negative real parts, and if $D\leq 0$ (it is possible only for $b<0$) --- two negative real roots $y_1$ and $y_2$ ($y_1=y_2$ if $D=0$).

According to Vi\`ete formulas $y_1+y_2=-y_0$, $y_1y_2=6b-y_0y_1-y_0y_2= 6b+y_0^2$ and then
$$
-2F(y) = \frac 16 y (y_0-y)(y^2+y_0y+y_0^2+6b).
$$
After the change of variable $y=y_0t$ first equation (\ref{11}) reduces to an equation
$$
I(b,c) \equiv \sqrt{6} \int_0^1 \frac{dt}{\sqrt{t(1-t)(y_0^2t^2+y_0^2t+y_0^2+6b)}} =1.
$$
It is easy to see that for the fixed $b$ the function $y_0(c)$ monotonically increases and $y_0(c)\to +\infty$ as 
$c\to +\infty$ (note that $y_0=\sqrt{8|b|}\cos\bigl(\frac13\arccos(3c/\sqrt{2|b|^3})\bigr)$ if $D<0$).
Then for the fixed $b$ the function $I(b,c)$ monotonically decreases and $\displaystyle\lim\limits_{c\to+\infty} I(b,c)=0$.
Moreover, if $c\to 0+0$ then $y_0(c)\to 0$ for $b\geq 0$ and $y_0(c)\to\sqrt{6|b|}$ for $b<0$. Therefore,
$$
\lim\limits_{c\to 0+0} I(b,c) =\frac1{\sqrt{b}}\int_0^1 \frac{dt}{\sqrt{t(1-t)}}=\frac\pi{\sqrt{b}}\quad \mathrm{if}\ b>0,
$$
$$
\lim\limits_{c\to 0+0} I(b,c) =\lim\limits_{c\to 0+0} \frac{\sqrt{6}}{\sqrt[3]{12c}} \int_0^1\frac{dt}{\sqrt{t(1-t^3)}} 
 =+\infty\quad \mathrm{if}\ b=0,
$$
$$
\lim\limits_{c\to 0+0} I(b,c) = \frac1{\sqrt{|b|}}\int_0^1 \frac{dt}{t\sqrt{1-t^2}} =+\infty\quad \mathrm{if}\ b<0.
$$
Hence, the desired positive value of $c$ exists and is unique if $b< \pi^2$, while for $b\geq \pi^2$ such a value does not exist.
\qed
\end{proof}

\begin{proof}[Theorem~\ref{T3}]
Equation (\ref{8}) is equivalent to equation
\begin{equation}\label{14}
y'' + by - \frac 13 y^3 =c
\end{equation}
for certain real constant $c$. Let
$$
F(y) \equiv -\frac1{12} y^4 + \frac{b}2 y^2 -cy = -\frac 1{12} y(y^3-6by+12c) \equiv -\frac1{12} y F_0(y).
$$
As in the proof of Theorem~\ref{T2} consider only the case $c>0$.

Again apply Cardano formulas. Let $D=-8b^3+36c^2$, 
$$
p= \sqrt[3]{-6c +\sqrt{D}},\quad q= \sqrt[3]{-6c -\sqrt{D}} \quad \mathrm{if}\ D\geq 0,
$$
$$
p= \sqrt[3]{-6c +i\sqrt{|D|}}= \sqrt{2b}e^{\frac{i}3\bigl(\pi+\arccos(3c/\sqrt{2b^3})\bigr)},\quad q=\overline{p} \quad \mathrm{if}\ D<0.
$$
If $D>0$ then the function $F_0$ has a real root $y_0=p+q<0$ and two complex conjugate roots $y_1$ and $y_2$. If $D=0$ then again the function $F_0$ has a real root $y_0=p+q<0$ and a double real root $y_1=y_2>0$. Both these two cases do not satisfy the hypothesis of Lemma~\ref{L2} since $F'(0)<0$.

It remains to consider the case $D<0$ (it is possible only if $b>0$), then $c\in (0,\frac{\sqrt{2}}3 b^{3/2})$. Here the function $F_0$ has three distinct real roots, where a root 
$y_0=p+q=\sqrt{8b}\cos\bigl(\frac{\pi}3+\frac13\arccos(3c/\sqrt{2b^3})\bigr)>0$, a root $y_1<0$, a root $y_2>y_0$. We have that $y_1+y_2=-y_0$, $y_2y_2=-6b+y_0^2$ and then
$$
-2F(y) = \frac 16 y (y_0-y)(6b-y_0^2-y_0y-y^2).
$$
After the change of variable $y=y_0t$ first equation (\ref{11}) reduces to an equation
$$
I(b,c) \equiv \sqrt{6} \int_0^1 \frac{dt}{\sqrt{t(1-t)(6b-y_0^2(1+t+t^2))}} =1.
$$
Similarly to the previous theorem for the fixed $b$ the function $y_0(c)$ monotonically increases, therefore, unlike to the previous theorem the function $I(b,c)$ also monotonically increases. It is easy to see that
$$
\lim\limits_{c\to 0+0} I(b,c) =\frac1{\sqrt{b}}\int_0^1 \frac{dt}{\sqrt{t(1-t)}}=\frac\pi{\sqrt{b}},
$$
$$
\lim\limits_{c\to \frac{\sqrt{2}}3 b^{3/2}-0} I(b,c) = \sqrt{\frac3b}\int_0^1 \frac{dt}{(1-t)\sqrt{t(t+2)}} =+\infty.
$$
Hence, the desired positive value of $c$ exists and is unique if $b> \pi^2$, while for $b\leq \pi^2$ such a value does not exist.
\qed
\end{proof}

\begin{remark}\label{R3}
In \cite{AnNa, An, Na} periodic solutions of equations (\ref{2}) and (\ref{3}) were considered in the case when the constant $c=0$ in equations (\ref{13}) and (\ref{14}). Therefore, the periodic solutions constructed in the present paper do not coincide with solutions from that papers. 
\end{remark}

\begin{acknowledgement}
The first author was supported by Project 333, State Assignment in the field of scientific activity implementation of Russia.
\end{acknowledgement}

\input{referenc}

\end{document}

%% file: referenc.tex
%
%
%